\documentclass[onefignum,onetabnum]{siamart220329}
\usepackage{epsfig,amssymb,amsfonts,mathrsfs}
\usepackage{color}
\usepackage{algorithm, algpseudocode}
\usepackage{subfig}
\usepackage{graphicx}
\newcommand{\intl}{\int\limits}
\newcommand{\suml}{\sum\limits}
\newcommand{\R}{\mathbb{R}}

\newcommand{\xb}{\mathbf{x}}
\newcommand{\yb}{\mathbf{y}}
\newcommand{\rb}{\mathbf{r}}
\newcommand{\vb}{\mathbf{v}}
\newcommand{\chib}{\boldsymbol{\chi}}
\newcommand{\phib}{\boldsymbol{\varphi}}
\newcommand{\Ic}{\mathcal{I}}

\newcommand{\Cc}{\mathcal{C}}
\newcommand{\Nbrs}{\mathcal{N}}

\renewcommand{\hat}{\widehat}
\renewcommand{\phi}{\varphi}
\renewcommand{\epsilon}{\varepsilon}
\def \abs#1{{\left|#1\right|}}
\def \Norm#1{{\|#1\|}}
\def \half{\frac{1}{2}}
\def \thrf{\frac{3}{2}}
\def \quart{\frac{1}{4}}

\begin{document}

\title{A hybrid interpolation ACA accelerated method for parabolic boundary integral operators\thanks{This work was created while J.T. enjoyed the hospitality of IIT Madras. }}

\author{Sivaram Ambikasaran\thanks{Wadhwani School of Data Science \& Artificial Intelligence, IIT Madras, Chennai, India \email{sivaambi@alumni.stanford.edu}}  
\and 
Ritesh Khan\thanks{Department of Mathematics, IIT Madras, Chennai, India \email{ma19d008@smail.iitm.ac.in}} 
\and 
Johannes Tausch\thanks{Southern Methodist University, Dallas, TX, USA \email{tausch@smu.edu}} 
\and 
Sihao Wang\thanks{Southern Methodist University, Dallas, TX, USA \email{sihaow@smu.edu}}
}
\maketitle

\begin{abstract}
  We consider piecewise polynomial discontinuous Galerkin discretizations of boundary integral reformulations of the heat equation.  The resulting linear systems are dense and block-lower triangular and hence can be solved by block forward elimination.  For the fast evaluation of the history part, the matrix is subdivided into a family of sub-matrices according to the temporal separation. Separated blocks are approximated by Chebyshev interpolation of the heat kernel in time. For the spatial variable, we propose an adaptive cross approximation (ACA) framework to obtain a data-sparse approximation of the entire matrix. We analyse how the ACA tolerance must be adjusted to the temporal separation and present numerical results for a benchmark problem to confirm the theoretical estimates.
\end{abstract}



%\subclass{35K20, 65F50, 65M38}
%\keywords{Boundary element methods, heat equation, sparse grids}

\section{Introduction}
The solution of parabolic problems by boundary integral techniques is a well-known alternative to finite element or finite difference methods, which has generated considerable interest in the engineering community \cite{rizzo-shippy71, wrobel-brebbia92, grigoriev-dargush02}.

To fix the main ideas, consider the single layer
potential
\begin{equation*}
  \mathcal{V} q(\xb,t) = \int_0^t\!\! \int_{\Gamma}
  G(\xb-\yb, t-\tau) q(\yb,\tau)\, ds(\yb) d\tau,
\end{equation*}
where $\Gamma$ is the boundary of a domain $\Omega \subset \R^3$,  $q$
is a density on $\Gamma$ and
\begin{equation*}
G(\rb,t) = \frac{1}{(4\pi t)^\thrf}
        \exp\left( - \frac{\abs{\rb}^2}{4 t} \right)
\end{equation*}
is the Green's function.
For any density $q$, the potential $\mathcal{V} q$ solves the heat equation
$\partial u = \Delta u$ in $\Omega$ with homogeneous initial condition. The
solution of the Dirichlet problem on $\Gamma$ can be obtained by
taking the boundary trace of $\mathcal{V} q$. This leads to the boundary 
integral equation
\begin{equation}\label{bie:sl}
  \mathcal{V} q(\xb,t) = f(\xb,t),\quad (\xb,t) \in \Gamma \times I,
\end{equation}
where $f$ is the Dirichlet datum and $I = [0,T]$ is the time interval 
where the solution is sought. The well-posedness of \eqref{bie:sl} hinges 
on the fact that the single layer potential $\mathcal{V}$ is coercive in the anisotropic 
Sobolev space $H^{-\half,-\quart}(\Gamma\times I)$, see
\cite{arnold-noon89,costabel90}. This result is the basis for the
error analysis of Galerkin discretization schemes of
\eqref{bie:sl}. 

Boundary integral formulations of more general boundary value problems can be derived 
from the Green's representation formula. However, we only focus on 
\eqref{bie:sl} in this article for brevity. The extension to treat the thermal double and hypersingular operators is straightforward and not discussed separately.

The time convolution in \eqref{bie:sl} implies that a numerical time stepping 
scheme requires evaluating an integral over the entire history of the previous 
time steps. Thus, the computational complexity scales quadratically in both space 
and time and thus, numerical methods for efficient time convolutions have been researched extensively.

One approach is to exploit the semigroup property of the heat equation. 
Here the layer potentials are evaluated only for a short time interval and then the 
method is restarted by evaluating the heat equation inside the domain. For 
representative papers in that direction, we refer to
\cite{greengard-lin99,li-greengard07,veerapaneni-biros08,greengard-wang19} and the cited references
therein.

An alternative is to evaluate the entire convolution in \eqref{bie:sl} by using fast methods which approximate the heat kernel
by a truncated kernel expansions in both space and time. This leads to low-rank
approximations of blocks in the discretization matrix that are well separated in 
time. In \cite{tausch06}, such an approach was first described in the context of
Nystr\"om discretizations, which was later extended to Galerkin discretizations 
in \cite{messner-etal14, messner-etal15} and space-time parallel implementations
\cite{watschinger-etal22}. 

One issue with methods that rely on degenerate series expansions
is that their formulation and convergence is strongly dependent on the particular
expansion. Here, the different scaling of the space and time variables in 
the heat kernel requires special care, which considerably complicates the algorithm, 
see \cite{messner-etal14,watschinger22}. 

A well-known alternative to kernel expansions is the adaptive cross approximation 
(ACA). This algorithm computes a low-rank approximation of a matrix block from
a sequence of sampled rows and columns. Its effectiveness depends on how well the 
kernel can be approximated by a truncated series, but does not require the knowledge
of a particular expansion. For integral operators that arise from elliptic PDEs,
this is by now a well-established method, with too many relevant
papers to review here. Instead, we refer to the books by Hackbusch and Bebendorf 
\cite{aca, borm-grase-hackb02, bebendorf08}. Because of its flexibility, the ACA has 
also found its use in software packages for general elliptic BEM methods,
see~\cite{betcke-etal15}. The ACA has also been applied recently for convolution 
quadrature discretizations of hyperbolic boundary integral operators, see 
\cite{seibel22, haider-rjasanow-schanz23}. In the context of thermal layer 
potentials, the ACA compression was applied in \cite{watschinger22} to compress the 
temporal near-field, by treating the first few steps in the time convolution as a 
sequence of elliptic operators. However, a space-time ACA method appears to be new.

The goal of this paper is to develop a hybrid approach for time dependent boundary
integral operators. Specifically, we will employ a combination of kernel 
interpolation in the time variable with the ACA compression for the spatial 
variable. This is a continuation of the work that was initiated in the PhD thesis
\cite{wang20}. In the present paper, we will extend the methodology to handle 
higher order discretizations in space and time. In addition, we will provide a 
detailed analysis of how the ACA approximation affects the overall error of the 
numerical solution. This will lead to a strategy to optimally choose the ACA 
accuracy depending on the temporal and spatial level of the block. 
 
\section{Discretization of Thermal Layer Potentials}
This section briefly discusses discontinuous piecewise polynomial ansatz functions in space and time and shows how higher order elements lead to a block structure in the matrix representation of the discrete potential. This viewpoint will facilitate solving the
integral equation formulation by block forward elimination as well as the generation 
of low-rank approximations by the ACA. It is important to remark that all thermal layer 
potentials are conforming to discontinuous temporal elements, but the double, adjoint and 
hypersingular potentials require continuity in space. The latter can be achieved by 
applying the restriction operators discussed in \Cref{sec:contin}.

\subsection*{Temporal Discretization}
The time discretization is based on a subdivision of the time interval $I=[0,t]$ 
into $N_t$ intervals $I_i = [t_{i-1},t_i]$. We assume that this subdivision is
uniform, i.e. $h_t = T/N_t$ and $t_i = ih_t$. Note that non-uniform time 
discretizations are also possible and have been considered in the context of 
parabolic layer potentials in \cite{watschinger-of23}. However, this will
increase the technical level of the presentation and will not be discussed 
here. 

The temporal ansatz space $S_I$ consists of functions that are polynomials of
degree $p_t$ in each
$I_i$. To ensure lower triangular block structure, we do not impose
continuity across intervals. Thus a basis of $S_I$ consists of
functions $\chi_{i1},\dots ,\chi_{iD_t}$ that form a polynomial basis in
$I_i$ and are zero elsewhere. Here $D_t = p_t + 1$ is the number of basis
functions per interval. Below, we will write these functions into a
vector
\begin{equation}\label{def:chib}
  \chib_i(t) =
  \begin{bmatrix}
      \chi_{i 1}(t)\\ \vdots \\\chi_{i D_t}(t)
    \end{bmatrix},\quad  t \in I \;\text{and}\; i=1\dots N_t.
\end{equation}
The dimension of $S_I$ is $D_t N_t$. For instance, in the case of piecewise linear
functions one could set 
\begin{equation*}
  \chib_i(t) =
  \begin{bmatrix}
       (t - t_{i-1})/h_t\\(t _i- t)/h_t 
    \end{bmatrix}.
\end{equation*}

\subsection*{Spatial Discretization}
The surface $\Gamma$ is divided into $N_s$ triangular patches
\begin{equation*}
  \Gamma = \bigcup_{k=1}^{N_s} \Gamma_k, 
\end{equation*}
where $\Gamma_k$ is the image of a parameterization from
the standard triangle into three dimensional space.  The spatial
ansatz space $S_\Gamma$ consists of functions whose pullback from
$\Gamma_k$ is a polynomial of degree $p_s$ in the standard
triangle.  We denote by $h_s$ the maximal diameter of the $\Gamma_k$'s and by
$\phi_{k 1},\ldots,\phi_{k D_s}$, $k = 1,\ldots,N_s$ the usual nodal
shape functions on patch $\Gamma_k$,
which are extended by zero on the remainder of the surface. Thus the
finite element space is
\begin{equation*}
    S_\Gamma = \mbox{span} \Big\{ \phi_{k m}, 1\leq k\leq N_s, 
    1\leq m \leq D_s \Big\}.
\end{equation*}
The nodes on patch $\Gamma_k$ are denoted by $\vb_{k m}$. Moreover,
$\phi_{k m}(\vb_{k' m'}) = \delta_{k,k'} \delta_{m m'}$.  For instance,
in the case of piecewise linear elements, the nodes are the three
vertices of the triangular patches and in the case of piecewise quadratic elements, the nodes are the vertices and the midpoints of the edges. Since the nodes on the edges are shared with the adjacent patches, there are repetitions in the
$\vb_{k m}$'s, and therefore it will be useful to introduce a unique index. Thus, the set of vertices has two representations  
\begin{equation*}
  \mathcal{V} = \{ \vb_{km} \,:\, 1\leq k\leq N_s, 1\leq m \leq D_s \} 
  = \{ \vb_{v} \,:\, 1\leq v\leq N_{v} \} .
\end{equation*}
To elucidate the block structure in discretization matrices, we will write the basis functions corresponding to a patch into a vector
\begin{equation}\label{def:phib}
  \phib_k(\xb) =
  \begin{bmatrix}
      \phi_{k 1}(\xb)\\ \vdots \\\phi_{k D_s}(\xb)
    \end{bmatrix},\quad \xb \in \Gamma \;\text{and}\; k=1,\dots, N_s,
\end{equation}
where $D_s$ denotes the number of functions per patch. 
As in the temporal discretization, the space $S_\Gamma$ is not
continuous across the patches. However, it has basis functions supported 
by only one patch. The latter will facilitate the compression of matrices 
by the ACA described later on.

\subsection*{Space-Time mesh} The space time discretization is
the cartesian product $S = S_\Gamma \times S_I$, which has a basis
that consists of all functions $\chi_{ij}\phi_{km}$. Using the vectors defined \eqref{def:chib} and \eqref{def:phib} a function in
$q \in S$ has the equivalent expansions
\begin{equation*}
  q(\xb,t) = \sum_{i=1}^{N_t} \sum_{j=1}^{D_t} \sum_{k=1}^{N_s} \sum_{m=1}^{D_s}
  q_{i j k m} \chi_{i j}(t) \phi_{k m} (\xb) 
  = \sum_{i=1}^{N_t} \sum_{k=1}^{N_s} \chib_i(t)^T Q_{i k} \phib_k(\xb),
\end{equation*}
where $Q_{i k} = [q_{i j k m}]_{j,m} \in \R^{D_t\times D_s}$.
Using the vector notations, it can be seen easily that the discretized thermal layer potentials have block structure. In particular, the matrix corresponding to the single layer potential decomposes into blocks $V_{i k , i' k'}$ of size $D_s D_t$
\begin{equation}\label{def:matV}
V_{i k , i' k'} =
  \intl_{I_i} \intl_{I_{i'}} \intl_{\Gamma_k}\intl_{\Gamma_{k'}} G(\xb-\yb,t-\tau)
  \chib_i(t)\chib_{i'}^T(\tau)\otimes\phib_k(\xb)\phib^T_{k'}(\yb)
  \,ds_\yb ds_\xb d\tau dt,
\end{equation}
where $\otimes$ denotes the Kronecker product and $\chib_i(t)\chib_{i'}^T(\tau)$
and $\phib_k(\xb)\phib^T_{k'}(\yb)$ are outer products. Agglomerating all spatial variables gives blocks of size  $N_s D_s D_t$
\begin{equation*}
V_{i i'} = \Big[ V_{i k , i' k'} \Big]_{k,k'}.
\end{equation*}
Note that $V_{i i'} = 0$ when $i'>i$, because $G(\rb,t-\tau)=0$ when
$\tau>t$.  Further, with the uniform subdivision of the time
interval, it follows that  $V_{i i'}$ only depends on the
difference $i-i'$ and hence there are matrices $A_d$ such that
$V_{i i'} = A_{i-i'}$. Thus the discrete
thermal potential is a block Toeplitz lower triangular matrix
\begin{equation}\label{toeplitz:uncompressed}
  \bf{V} = \begin{bmatrix}
    A_0\\
    A_1 & A_0 \\
    \vdots &       & \ddots\\
    A_{N_t} & \hdots  &       & A_0
    \end{bmatrix}.
\end{equation}

This structure suggests to solve the linear system $\bf{V}q = p$
by block wise forward elimination, where the $i$-th elimination
step is
\begin{equation}\label{forward:elimination}
\begin{aligned}
  b_i &= p_i - \sum_{i'=1}^{i-1}  A_{i-i'} q_{i'}\,,\\
  q_i &= A_0^{-1}b_i  \,.
\end{aligned}
\end{equation}
Note that the matrices $A_{d}$ are still large because they contain
the entire spatial dependence.  Because of the exponential decay of
the Green's function in space, one can expect that the matrices are
sparse when $d$ is small. However, they fill up as the difference
of the indices gets larger.

The linear system for $q_i$ can be solved iteratively by exploiting the
sparsity, we will give some more details in \Cref{sec:numresults}
below. Most of the computational work is spent evaluating
the right hand sides $b_i$. For this task, we will use a completely
different algorithm from naively evaluating the sum in 
\eqref{forward:elimination}. This will be described in the following
section.

\section{Hybrid ACA Method for the Discrete Thermal Layer Potential}\label{sec:aca}
\subsection*{Hierarchical partitioning of matrices and vectors} 
We first generate a hierarchical block partition of the system matrix ${\bf V}$ in
the time variable. The singularity of the Green's function occurs when
$t=\tau$ which appears in the diagonal and subdiagonal blocks of ${\bf V}$. As the
difference $t-\tau$ is increased, the kernel gets smoother and hence blocks of 
${\bf V}$ farther away from the diagonal can be better approximated by low-rank
matrices. This suggests the partitioning shown in \Cref{fig:H-matrix}. 
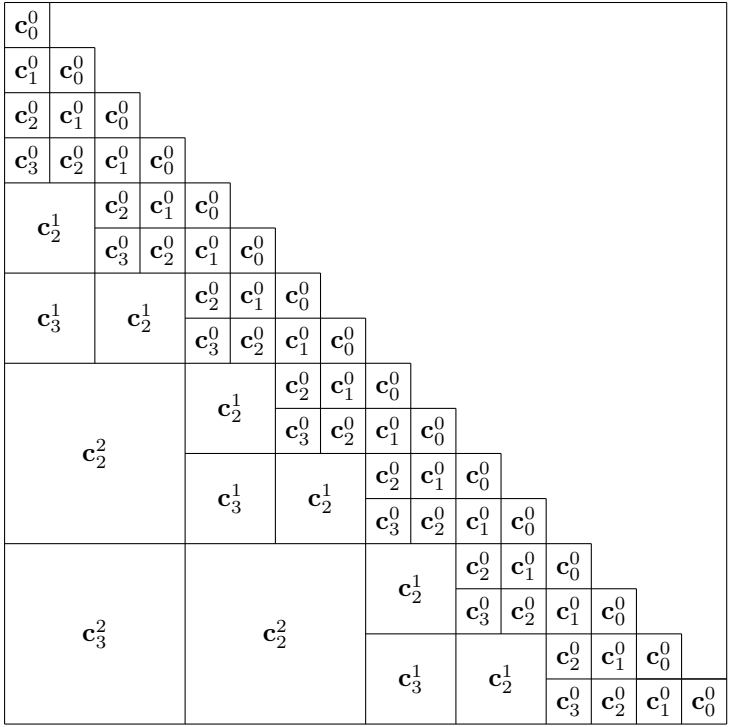
\begin{figure}
\begin{center}
\setlength{\unitlength}{0.6cm}
\begin{picture}(16,16)
  \put(0,0){\line(1,0){16}}
  \put(0,4){\line(1,0){13}}
  \put(0,8){\line(1,0){9}}
  \put(0,16){\line(1,0){16}}
  \put(0,0){\line(0,1){16}}
  \put(4,0){\line(0,1){13}}
  \put(8,0){\line(0,1){9}}
  \put(16,0){\line(0,1){16}}

  \put(0,10){\line(1,0){7}}
  \put(0,12){\line(1,0){5}}
  \put(0,14){\line(1,0){3}}
  \put(4,6){\line(1,0){7}}
  \put(8,2){\line(1,0){7}}
  \put(10,0){\line(0,1){7}}
  \put(12,0){\line(0,1){5}}
  \put(14,0){\line(0,1){3}}
  \put(6,4){\line(0,1){7}}
  \put(2,8){\line(0,1){7}}

  \put(1,12){\line(0,1){4}}
  \put(3,10){\line(0,1){4}}
  \put(5,8){\line(0,1){4}}
  \put(7,6){\line(0,1){4}}
  \put(9,4){\line(0,1){4}}
  \put(11,2){\line(0,1){4}}
  \put(13,0){\line(0,1){4}}
  \put(15,0){\line(0,1){2}}

  \put(14,1){\line(1,0){2}}
  \put(12,1){\line(1,0){4}}
  \put(10,3){\line(1,0){4}}
  \put(8,5){\line(1,0){4}}
  \put(6,7){\line(1,0){4}}
  \put(4,9){\line(1,0){4}}
  \put(2,11){\line(1,0){4}}
  \put(0,13){\line(1,0){4}}
  \put(0,15){\line(1,0){2}}
  \put(0.5,15.5){\makebox(0,0){${\bf c}_0^0$}}
  \put(1.5,14.5){\makebox(0,0){${\bf c}_0^0$}}
  \put(2.5,13.5){\makebox(0,0){${\bf c}_0^0$}}
  \put(3.5,12.5){\makebox(0,0){${\bf c}_0^0$}}
  \put(4.5,11.5){\makebox(0,0){${\bf c}_0^0$}}
  \put(5.5,10.5){\makebox(0,0){${\bf c}_0^0$}}
  \put(6.5,9.5){\makebox(0,0){${\bf c}_0^0$}}
  \put(7.5,8.5){\makebox(0,0){${\bf c}_0^0$}}
  \put(8.5,7.5){\makebox(0,0){${\bf c}_0^0$}}
  \put(9.5,6.5){\makebox(0,0){${\bf c}_0^0$}}
  \put(10.5,5.5){\makebox(0,0){${\bf c}_0^0$}}
  \put(11.5,4.5){\makebox(0,0){${\bf c}_0^0$}}
  \put(12.5,3.5){\makebox(0,0){${\bf c}_0^0$}}
  \put(13.5,2.5){\makebox(0,0){${\bf c}_0^0$}}
  \put(14.5,1.5){\makebox(0,0){${\bf c}_0^0$}}
  \put(15.5,0.5){\makebox(0,0){${\bf c}_0^0$}}
  \put(0.5,14.5){\makebox(0,0){${\bf c}_1^0$}}
  \put(1.5,13.5){\makebox(0,0){${\bf c}_1^0$}}
  \put(2.5,12.5){\makebox(0,0){${\bf c}_1^0$}}
  \put(3.5,11.5){\makebox(0,0){${\bf c}_1^0$}}
  \put(4.5,10.5){\makebox(0,0){${\bf c}_1^0$}}
  \put(5.5,9.5){\makebox(0,0){${\bf c}_1^0$}}
  \put(6.5,8.5){\makebox(0,0){${\bf c}_1^0$}}
  \put(7.5,7.5){\makebox(0,0){${\bf c}_1^0$}}
  \put(8.5,6.5){\makebox(0,0){${\bf c}_1^0$}}
  \put(9.5,5.5){\makebox(0,0){${\bf c}_1^0$}}
  \put(10.5,4.5){\makebox(0,0){${\bf c}_1^0$}}
  \put(11.5,3.5){\makebox(0,0){${\bf c}_1^0$}}
  \put(12.5,2.5){\makebox(0,0){${\bf c}_1^0$}}
  \put(13.5,1.5){\makebox(0,0){${\bf c}_1^0$}}
  \put(14.5,0.5){\makebox(0,0){${\bf c}_1^0$}}
  \put(0.5,13.5){\makebox(0,0){${\bf c}_2^0$}}
  \put(1.5,12.5){\makebox(0,0){${\bf c}_2^0$}}
  \put(2.5,11.5){\makebox(0,0){${\bf c}_2^0$}}
  \put(3.5,10.5){\makebox(0,0){${\bf c}_2^0$}}
  \put(4.5,9.5){\makebox(0,0){${\bf c}_2^0$}}
  \put(5.5,8.5){\makebox(0,0){${\bf c}_2^0$}}
  \put(6.5,7.5){\makebox(0,0){${\bf c}_2^0$}}
  \put(7.5,6.5){\makebox(0,0){${\bf c}_2^0$}}
  \put(8.5,5.5){\makebox(0,0){${\bf c}_2^0$}}
  \put(9.5,4.5){\makebox(0,0){${\bf c}_2^0$}}
  \put(10.5,3.5){\makebox(0,0){${\bf c}_2^0$}}
  \put(11.5,2.5){\makebox(0,0){${\bf c}_2^0$}}
  \put(12.5,1.5){\makebox(0,0){${\bf c}_2^0$}}
  \put(13.5,0.5){\makebox(0,0){${\bf c}_2^0$}}
  \put(0.5,12.5){\makebox(0,0){${\bf c}_3^0$}}
  \put(2.5,10.5){\makebox(0,0){${\bf c}_3^0$}}
  \put(4.5,8.5){\makebox(0,0){${\bf c}_3^0$}}
  \put(6.5,6.5){\makebox(0,0){${\bf c}_3^0$}}
  \put(8.5,4.5){\makebox(0,0){${\bf c}_3^0$}}
  \put(10.5,2.5){\makebox(0,0){${\bf c}_3^0$}}
  \put(12.5,0.5){\makebox(0,0){${\bf c}_3^0$}}
  \put(1,11){\makebox(0,0){${\bf c}_2^1$}}
  \put(3,9){\makebox(0,0){${\bf c}_2^1$}}
  \put(5,7){\makebox(0,0){${\bf c}_2^1$}}
  \put(7,5){\makebox(0,0){${\bf c}_2^1$}}
  \put(9,3){\makebox(0,0){${\bf c}_2^1$}}
  \put(11,1){\makebox(0,0){${\bf c}_2^1$}}
  \put(1,9){\makebox(0,0){${\bf c}_3^1$}}
  \put(5,5){\makebox(0,0){${\bf c}_3^1$}}
  \put(9,1){\makebox(0,0){${\bf c}_3^1$}}
  \put(2,6){\makebox(0,0){${\bf c}_2^2$}}
  \put(6,2){\makebox(0,0){${\bf c}_2^2$}}
  \put(2,2){\makebox(0,0){${\bf c}_3^2$}}
\end{picture}
\caption{\label{fig:H-matrix}Partitioning of 
 ${\bf V}$
 for the case of four temporal levels.}
\end{center}
\end{figure}
This partitioning can be more formally described by a hierarchy of
time intervals. To that end, divide the interval
$[0,T]$ two equal subintervals and then recursively divide each subinterval
until the nodes in the finest level contain a predetermined (small) number
of time steps $n_T$. The relationship between $N_t$, $n_T$ and the number of refinements $L$ is
\begin{equation*}
    N_t = 2^L n_T.
\end{equation*}
Moreover, the number of temporal degrees of freedom in the discretization 
is $N_t D_t$. 
The refinement scheme results in a binary tree of intervals with nodes
$$
I_n^\ell = \frac{T}{2^{L-\ell}} [n,n+1), \quad n\in\{0,\dots,2^{L-\ell}-1\}.
$$
Mind that $\ell=0$ denotes the finest and $\ell=L$ the coarsest level. We denote by $\Ic_n^\ell$ the set of indices
We denote by $\Ic_n^\ell$ the set of indices
of the temporal basis functions that are supported in $I_n^\ell$, that
is,
\begin{equation*}
\Ic_n^\ell = \left\{ i \,:\, I_i \subset I_n^\ell \right\}.
\end{equation*}  
Note that $N_{\ell,t} := \# \Ic_n^\ell$ is the same for all $n$.
The block ${\bf c}^{\ell}_d$ of the matrix ${\bf V}$ contains the entries
\begin{equation}\label{def:cBlock}
\left[ {\bf c}^{\ell}_d \right]_{i k, i',k'} = V_{i k, i' k'},\quad
i\in \Ic_d^\ell, \; i'\in  \Ic_0^\ell, \;\text{and}\; k,k' \in \{1,\dots,N_s\}.
\end{equation}

Because of the hierarchical block Toeplitz structure of ${\bf V}$, the only
distinct blocks are
${\bf c}^{\ell}_2$ and  ${\bf c}^{\ell}_3$, $\ell\in \{2,\dots,L\}$
and 
${\bf c}^{L}_0$ and  ${\bf c}^{L}_1$. Since the superscript $d$ indicates
the relative separation of the time variables, the former blocks will be 
referred to as the temporal far-field, and the latter blocks as the temporal 
near-field.

To describe the block elimination algorithm below, we partition vectors as
\begin{equation}\label{def:qBlock}
  {\bf q}_{\ell}^n = \left[ q_{i,k}\right]_{i\in \Ic_{\ell}^n, 1\leq k\leq N_s},
\end{equation}
where $n \in \{0,\dots, 2^{L-\ell}-1\}$ and $\ell \in \{0,\dots,L\}$.

\subsection*{Block Forward Elimination Algorithm}
We now turn to solve the system ${\bf V}{\bf q} = {\bf p}$ using a hierarchical
version of block-forward elimination which is based on the partition of ${\bf V}$ 
in \Cref{fig:H-matrix}. Here, in the $n$-th time step, the partially
computed vector ${\bf q}$ is needed to find the components of {\bf q} that belong
to this time step. To that end, consider the base-2 expansion of $n$
$$
n = (\sigma_R\dots \sigma_1 \sigma_0)_2 
= \sigma_R 2^R + \dots + 2 \sigma_1 + \sigma_0
$$
where $\sigma_i \in \{ 0,1\}$, $0\leq i < R$ and $\sigma_R = 1$.
The following notations will be useful later on
\begin{itemize}
\item[] $R$: Index of the highest digit in the binary representation
  of $k$.
\item[] $S$: Index of the lowest non-zero digit in the binary representation
  of $k$.
\end{itemize}
For example, the base-2 representation of the number 20 is $(10100)_2$
which implies that $R=4$ and $S=2$.

The binary subdivision scheme of the time interval implies that
the parent of the finest level interval $I_{n}^0$ in level $\ell$ is
$I_{n_\ell}^\ell$ where $n_\ell$ is given by
\begin{equation*}
n_\ell = (\sigma_R\dots \sigma_{\ell+1} \sigma_\ell)_2
= \sigma_R 2^{R-\ell} + \dots + 2 \sigma_{\ell+1} + \sigma_\ell
\end{equation*}

With these notations, the hierarchical block forward elimination
algorithm is described in \Cref{algo:bForwardElim}. In this
algorithm, we set ${\bf p}_{n}^\ell = {\bf 0}$ whenever $n<0$. To
understand the algorithm, it is important to keep in mind that ${\bf p}_{n}^\ell$
signifies a certain portion of the vector ${\bf p}$, that is not
necessarily distinct for different $n$ and $\ell$. Thus, a change of 
${\bf p}_{n}^0$ will result in the same change of the parents
${\bf p}_{n_\ell}^\ell$ in all levels.

\begin{algorithm}
\caption{Hierarchical Block Forward Elimination Algorithm to solve
  ${\bf V}{\bf q} = {\bf p}$.}
\label{algo:bForwardElim}
\begin{algorithmic}
\For{ $n_0=0: 2^L-1$}

    \State Compute $S$ in the binary representation of $n_0$.
    \State ${\bf p}_{n_\ell}^\ell \leftarrow
               {\bf p}_{n_\ell}^\ell - {\bf c}_{1}^\ell\, {\bf q}_{n_\ell-1}^\ell$
   \Comment{Subtract Near-field.}

    \For{ $\ell=0:S$}\Comment{Subtract Far-field.}
        \State Compute $n_\ell$ from $n_0$.
        \State ${\bf p}_{n_\ell}^\ell \leftarrow
               {\bf p}_{n_\ell}^\ell - {\bf c}_{2}^\ell\, {\bf q}_{n_\ell-2}^\ell$
    \EndFor
    \State ${\bf p}_{n_S}^S \leftarrow
           {\bf p}_{n_S}^S - {\bf c}_{3}^\ell\, {\bf q}_{n_S-3}^\ell$\\
 
   \State Solve ${\bf c}_{0}^0\, {\bf q}_{n_0}^0 = {\bf p}_{n_0}^0$
\Comment{Diagonal block.}
\EndFor
\end{algorithmic}
\end{algorithm}

\subsection*{Low-rank approximation for the temporal far-field}
The temporal far-field consists of the blocks in $\mathbf{V}$ that have well 
separated time variables, which are the $\mathbf{c}^\ell_d$'s with $d\geq 2$.
As it is apparent from \Cref{fig:H-matrix}, they contribute to the dominant 
computational cost of the forward elimination algorithm.
If these matrices are evaluated directly, then \Cref{algo:bForwardElim} has the same complexity as the
standard block forward elimination method in
\eqref{forward:elimination}.

We now describe a fast method to evaluate the forward elimination
which is based on the separation of the temporal variables.
For $(t,\tau) \in I^\ell_d \times I^\ell_0$ and $d\geq 2$ the heat 
kernel is a smooth function in all variables. Hence, it can be 
interpolated by
\begin{equation}\label{interpol:G}
  G(\xb-\yb,t-\tau) \approx
  \sum\limits_{\beta=0}^{p-1} \sum\limits_{\beta'=0}^{p-1} 
  G(\xb-\yb,t_\beta-\tau_{\beta'}) L_\beta(t) L_{\beta'}(\tau)   
\end{equation}
where $t_\beta, \tau_{\beta'}$ are the interpolation nodes and
$L_\beta(t)$, $L_{\beta'}(\tau)$ the corresponding Lagrange
polynomials. Here, we use Chebyshev nodes, which for the interval $I^\ell_0$
are given by
\begin{equation*}
  \tau_\beta = T 2^{-\ell} \left( \half + 
    \half \cos\left( \frac{\pi}{2} \frac{2\beta+1}{p} \right) \right)
\end{equation*}
The interpolation error can be bounded independently of the spatial
difference $\xb-\yb$. A detailed analysis of this error can be found
in \cite{tausch12a}.

Substitution of \eqref{interpol:G} into \eqref{def:matV} shows that
\begin{equation*}
  \Big[{\bf c}^\ell_{d}\Big]_{i j k m, i' j' k' m'} =
  \sum\limits_{\beta=0}^{p-1} \sum\limits_{\beta'=0}^{p-1}
  M_{\beta, i j} A_{\beta k m, \beta' k' m'} M_{\beta', i' j'}
\end{equation*}
where $i\in \Ic^\ell_d$, $i'\in \Ic^\ell_0$, $1\leq j \leq D_t$, 
$1\leq m\leq D_s$ and 
\begin{align}
  M_{\beta, i j} &= \int_{I^\ell_0} L_\beta(t) \chi_{i j}(t) \,dt
  \label{def:Mmat}\\
  A_{\beta k m,\beta' k' m'} &= 
     \intl_{\Gamma}\intl_{\Gamma} G(\xb-\yb,t_\beta-\tau_{\beta'})
                              \phi_{k m}(\xb)\phi_{k' m'}(\yb) \,ds_\yb ds_\xb.
  \label{def:Afmat}                              
\end{align}
For the coefficients of the matrix
vector product
${\bf p}^\ell_{n+d} = {\bf c}^\ell_{d} {\bf q}^\ell_{n}$ it follows
that after rearranging the order of summations that
\begin{equation*}
  p_{i j k m} = \suml_{\beta} M_{\beta,i j}
           \suml_{\beta',k',m'} A_{\beta k m, \beta' k' m'}
           \suml_{i',j'} M _{\beta',i' j'} q_{i' j' k' m'}.
\end{equation*}
Thus, the sequence of calculating the product is
\begin{equation}\label{fast:time:algor}
\begin{aligned}
  m_{\beta' k' m'} &= \suml_{i' j'} M _{\beta',i' j'} \, q_{i' j' k' m'},\\
  u_{\beta k m} &= \suml_{\beta',k',m'}  
                A_{\beta k m, \beta' k' m'} \, m_{\beta' k' m'},\\
  p_{i j k m} &= \suml_{\beta} M _{\beta,i j} \, u_{\beta k m}.
\end{aligned}
\end{equation}
From this calculation, it also follows that the following relationship
holds
\begin{equation}\label{c:m:a}
  {\bf c}^\ell_{d} = (M^T_\ell \otimes I )\, A^\ell_{d}\, (I \otimes M_\ell).
\end{equation}
Here $A^\ell_d$ is the matrix with coefficients in \eqref{def:Afmat}, 
with the $\beta k m$-indices are flattened into a linear index, thus
$A^\ell_{d}$ is a square matrix of size $N_s D_s p$. Moreover,
$M_\ell$ is the matrix with coefficients in \eqref{def:Mmat}, and 
$I=I_{D_s N_s}$.

\section{ACA compression for the spatial blocks}
The dominant cost of evaluating the thermal layer potential via the
algorithm of \eqref{fast:time:algor} is the second step. We write this
as a matrix-vector product
\begin{equation*}
{\bf u}^\ell_{n} = A^\ell_{d} {\bf m}^\ell_{n-d}.
\end{equation*}
To obtain an efficient algorithm, the matrix-vector product will be
accelerated using a block wise low-rank approximation for the matrix 
$A^\ell_{d}$. This section describes the process in more detail.

To that end, consider a cube that fully contains the surface $\Gamma$. This 
cube is uniformly refined until the finest cubes contain a predetermined maximal 
number of patches. The cubes in the $\ell$-th level are denoted by $C^\ell_\nu$, 
where $\ell=0$ is the finest level, $\ell=L_s$ is the level of the initial 
cube, and $\nu$ is an index for all nonempty cubes in level $\ell$. 
Further, $\Cc^\ell_\nu$ denotes the index set of
$\Gamma_k$'s with centroid in $C^\ell_\nu$. 

We define the separation ratio of two different cubes in the same level as
\begin{equation}\label{def-sepRatio}
\eta(\nu,\nu') = \frac{\rho_{\nu'} + \rho_{\nu'}}{\abs{\xb_\nu - \xb_{\nu'}}}\,,
\quad \nu \not= \nu',
\;\mbox{and}\;\eta(\nu,\nu) = \infty.
\end{equation}
Here $\rho_\nu$ is the diameter and $\xb_\nu$ the centroid of all patches in 
$\Cc^\ell_\nu$.

The neighbors of a cube are the cubes in the same level 
for which the separation ratio is greater than a predetermined
constant  $\eta_0 < 1$. That is,
\begin{equation*}
  \Nbrs(\nu) = \{ \nu'  \,:\, \eta(\nu,\nu') > \eta_0  \}.
\end{equation*}
Recall that the heat kernel decays exponentially in space at a rate
that depends on the time difference $t-\tau$. Because of the scaling
of the space and time variables in the kernel, it suffices to consider
neighboring interactions if the spatial level $\ell_s$ depends on the temporal
level $\ell$ as follows
\begin{equation} \label{level_relation}
  \ell_s =  \min\left\{ \mbox{floor}\left(\frac{\ell}{2}\right), L_s \right\}
\end{equation}
where $\mbox{floor}$ rounds down to the next nearest integer. In \cite{tausch06},
it was shown that with this scaling of the temporal and spatial levels 
\begin{enumerate}
\item Interactions of cubes in the same level can be neglected unless
  they are neighbors. The resulting error only depends on $\eta_0$, but
  not on the level.
\item For $\nu,\nu'$ neighbors, and $\xb \in C_\nu^\ell$, $\yb \in
  C_{\nu'}^\ell$ the kernels $G(\xb,\yb,t_\beta - \tau_{\beta'})$
  can be approximated by polynomials in $\xb$ and $\yb$, with error
  bounds independent of the level $\ell$.
\end{enumerate}
Thus, the matrix $A^\ell_{d}$ is subdivided into blocks
$A^\ell_{d}(\nu,\nu')$, with coefficients
\begin{equation*}
A^\ell_{d}(\nu,\nu') = \Big[ A_{\beta k m,\beta' k' m'} \Big]_{
  k \in \Cc_\nu^{\ell_s}, k'\in\Cc_{\nu'}^{\ell_s} \atop
  {1 \leq m,m' \leq D_s \atop 0\leq \beta,\beta' < p}} , \quad \nu'\in \Nbrs(\nu)
\end{equation*}
and we set $A^\ell_{d}(\nu,\nu') = 0$ when $\nu$ and $\nu'$ are not neighbors. 

Since the nonzero blocks $A^\ell_{d}(\nu,\nu')$ are still large in the coarser spatial levels, we use the well-known adaptive cross approximation to obtain the approximating low-rank matrix factorization
\begin{equation}\label{aca:farfield}
A^\ell_{d}(\nu,\nu') = U V^T + E^\ell_{d}(\nu,\nu').
\end{equation}
Since the heat kernel in the $A^\ell_{d}(\nu,\nu')$'s is smooth, the inner 
dimension $r$ in the $U V^T$ product can be much smaller than the dimensions of
the original matrix, while still guaranteeing a small error. A more detailed 
analysis will be presented in \Cref{sec:analysis}.
The computation of $U$ and $V$ only involves $r$ rows and columns of 
of $A^\ell_{d}(\nu,\nu')$. A detailed description of the ACA factorization can 
be found, e.g., in \cite{bebendorf08}, Algorithm 3.1.

\subsection*{Temporal Near-field} We now turn to the temporal near-field of 
$\mathbf{V}$, which consists of the matrices $\mathbf{c}_0^0$ and  $\mathbf{c}_1^0$.  These blocks contain the singularity of the heat kernel and only a small number of timesteps. Therefore only the spatial variable will be compressed as described below. The explicit form of the near-field blocks is
\begin{equation*}
  \mathbf{c}_0^0 = \begin{bmatrix}
    A_0\\
    A_1 & A_0 \\
    \vdots & \ddots & \ddots \\
    A_{n_T-1} & \hdots & A_1 & A_0
  \end{bmatrix}
  \quad\mbox{and}\quad
  \mathbf{c}_1^0 = \begin{bmatrix}
    A_{N_t} & \dots & \dots &  A_{1}\\
    A_{N_t+1} & \ddots &    &  A_{2}\\
    \vdots & \ddots & \ddots & \vdots \\
    A_{2n_T-1} & \hdots & A_1 & A_{N_t}
  \end{bmatrix}
\end{equation*}
where the coefficients of $A_d$ are  
\begin{equation}\label{def:Anmat}  
  \Big[A_d\Big]_{k m j, k' m' j'} = 
     \intl_{\Gamma_k}\intl_{\Gamma_k} G_{d j j'}(\xb-\yb)
                              \phi_{k m}(\xb)\phi_{k' m'}(\yb) \,ds_\yb ds_\xb
\end{equation}
and
\begin{equation*}
  G_{d j j'}(\rb) = 
     \intl_{I^0_d}\intl_{I^0_0} G(\rb, t-\tau) \chi_j(t) \chi_{j'}(\tau)\,d\tau dt
\end{equation*}
is the time integrated kernel. Since the $\chi_j$'s are polynomials, 
the functions $G_d$ can be explicitly expressed in terms of
exponential and error functions, see, e.g., \cite{messner-etal14}.

For $d=0$ this kernel has a
$|\rb|^{-1}$ singularity,  $d=1$ the gradient has a $|\rb|^{-1}$
singularity, and for  $d\geq 2$ the kernel is smooth. In all cases of 
$0\leq d < 2 n_T$ there is exponential decay in $|\rb|$ and thus, for 
the finest level cluster $\nu$, only the neighbors $\Nbrs(\nu)$ contribute 
significantly to the matrix vector product. All remaining cubes can be neglected.

For the temporal far-field, the spatial meshwidth $h_s$ does not affect the asymptotic computational complexity of the method. However, in the near-field calculation, it does. It is straightforward to see that if $h_t \lesssim h_s^2$ the number of elements in a finest level cluster is independent of $h_s$, hence the evaluation of $A_d(\nu,\nu')$ can be done directly, without any ACA acceleration. Otherwise, if the spatial mesh is highly refined, the spatial clustering has to be
further refined, and the matrices have to be evaluated using a
hierarchical matrix approach similar to what is typically done for  the
$1/|\mathbf{r}|$ kernel in potential theory.

\section{Continuous spatial discretizations}\label{sec:contin}
So far, we have described the method for discontinuous ansatz
spaces which are suitable for single layer potential . However, to treat more general  boundary integral equations of the heat equation, elements with continuity in space are required. Hence, we describe in this section how to accomplish the ACA compression of the matrices 
\eqref{def:Afmat} and \eqref{def:Anmat} for  continuous elements. Note that continuity in time is not required  for thermal potentials, see, e.g., \cite{costabel90}. 

The space of continuous functions in $S_\Gamma$ is a subspace and
is denoted by $S_\Gamma^{cont}$. The nodal basis functions, denoted by
$\phi_v$, satisfy $\phi_v(\vb_{v'}) = \delta_{v,v'}$ for all vertices 
$v,v' \in \{1,\dots, N_v\}$ in the nodal basis. Since the $\phi_v$'s 
are continuous, they have support over several triangular patches.

When setting up a discrete layer operator, one
usually computes the integrals panel wise for the $D_s^2$ different
combinations of shape functions in the $\xb$ and $\yb$
variables. Then these values are added into the appropriate coefficients of the
matrix. This approach avoids repeated integration over patches that
would occur if the matrix was computed one coefficient at a time, by integrating 
over the entire support of each nodal basis function. Unfortunately, in the 
ACA compression, individual rows and columns have to be computed at a time, which 
can only be done by repeatedly integrating over the same pairs of patches.

This problem can be avoided by computing the ACA compression for the larger matrix 
in the discontinuous basis and then applying extension and
restriction operators to switch between continuous and
discontinuous spaces. To define an extension operator from $S_\Gamma$ to $S_\Gamma^{cont}$ consider 
the matrix $E \in \R^{D_s N_s \times N_v}$, with coefficients
\begin{equation}\label{def:Emat}
  E_{k m,v} = \begin{cases}
    1  & \text{if } \vb_{km} = \vb_v\,,\\
    0 & \text{else.}
  \end{cases}
\end{equation}
For the basis function $\phi_v \in S_\Gamma^{cont}$ it follows that
\begin{equation}\label{phiv:km}
  \phi_v(\xb) = \sum\limits_{\{ k,m: \vb_{km} = \vb_v\} } \!\!\! \! \phi_{km}(\xb)
              = \sum_{k=1}^{N_s} \sum_{M=1}^{D_s}   E_{km,v} \phi_{km}(\xb).
\end{equation}
If  $G(\xb,\yb)$ is one of the kernels in 
\eqref{def:Afmat} and \eqref{def:Anmat}, then the corresponding
matrix in the continuous matrix is given by
\begin{equation}\label{Avv}
  A^{cont}_{v,v'} = \intl_{\Gamma}\intl_{\Gamma} G(\xb-\yb)
                              \phi_v(\xb)\phi_{v'}(\yb) \,ds_\yb ds_\xb
\end{equation}
Substitution of \eqref{phiv:km} into \eqref{Avv} shows that
\begin{equation}\label{A:cont:discont}
  A^{cont} = E^T A E ,
\end{equation}
where $A$ is either $A_d^\ell$ or $A_d$ whose coefficients are defined
in \eqref{def:Afmat} or \eqref{def:Anmat}, respectively. A simple
calculation also shows that
\begin{equation}\label{EtEDv}
  E^T E = D_v,
\end{equation}
where $D_v \in \R^{N_v\times N_v}$ is a diagonal matrix that contains
the degrees of the nodes.

We now return to the $i$-th elimination step in the forward
elimination. If in \eqref{forward:elimination} $A_d$ is replaced by the matrix of
the continuous basis function, then it follows from
\eqref{A:cont:discont} and \eqref{EtEDv} that
\begin{align*}
  A_0^{cont} q_i &= p_i  - 
     \sum_{i'=1}^{i-1} E^T A_{i-i'} E q_{i'}\,, \\
           &= E^T \left( \hat p_i - \sum_{i'=1}^{i-1} A_{i-i'} \hat q_{i'}\right) ,
\end{align*}
where $\hat p_i = E D_v^{-1} p_i$ and  $\hat q_{i} = E q_{i}$. This
implies that the solution with continuous elements can be computed
using only matrices with discontinuous elements if in
\eqref{forward:elimination} the right hand side is replaced by $\hat p_i$
and the vectors $q_{i'}$ are replaced by $\hat q_{i'}$. Moreover, the linear system with $A_0^{cont}$ is solved by iteration where the matrix-vector product is evaluated using \eqref{A:cont:discont}.

\section{Analysis}\label{sec:analysis}
In the above algorithm, low-rank approximations of matrix blocks are generated in the near-field as well as the far-field for different temporal levels. We now turn to the analysis of the error introduced
by these approximations and the resulting asymptotic complexity.  The
goal is to optimize the choice of the ACA truncation parameter for the
type and level of the matrix block.

While the algorithm described above is applicable to tensor products of fairly general spatial and temporal meshes, we consider here the case where an initially coarse mesh is refined uniformly such that the temporal meshwidths satisfy the asymptotic relation
$h_t \sim h_s^2$. That is, in a refinement step, each spatial patch and
each temporal interval are refined into four pieces. While this
assumption is somewhat restrictive, it keeps the technical level of
the ensuing discussion at a minimum.

With the uniform refinement scheme, it follows that the maximal
number of patches $\Gamma_k$ associated with the cube $\nu$ in spatial
level $\ell_s$ satisfies the asymptotic estimate
\begin{equation}\label{est:patches}
N_s(\nu) \lesssim  4^{\ell_s} = 2^{\ell}.
\end{equation}
Here, the second step follows form \eqref{level_relation} and $\lesssim$ means that the inequality holds up to a constant independent of the number of mesh refinements. The number of nonempty cubes in a given level is bounded by
\begin{equation}\label{est:cubes}
\# \Cc_{\ell_s} \lesssim  N_s 4^{-\ell_s} = N_s 2^{-\ell}.
\end{equation}
For ensuing analysis  we need the following basic estimates. The next lemma is useful for matrices that can be partitioned in a way such that most blocks are zero matrices. That is, the number of nonzero blocks per row and column, given by
\begin{equation*}
\gamma_C = \max_{k} \#\{ \ell: B_{k\ell} \not= 0 \}
\quad\mbox{and}\quad
\gamma_R = \max_{\ell} \#\{ k: B_{k\ell} \not= 0 \}
\end{equation*}
are small.

\begin{lemma}\label{lem:blocksparse}
If $A = [B_{k \ell}]_{k,\ell}$, where
$B_{k \ell}\in \R^{m_k \times n_\ell}$, then the following estimate holds
\begin{equation*}
\Norm{A}_2 \leq \gamma_C^\half \gamma_R^\half 
\max_{k,\ell} \Norm{B_{k \ell}}_2.
\end{equation*}
\end{lemma}
\begin{proof}
Let the vectors $x$ and $y$ be partitioned into subvectors $x_k\in
\R^{m_k}$ and  $y_\ell\in \R^{n_\ell}$, then 
\begin{align*}
  \abs{x^T A y } 
  &= \abs{ \sum_{(k,\ell):B_{k \ell}\not=0} x_k^T B_{k \ell} y_\ell }  
   \leq \sum_{(k,l): B_{k \ell}\not=0} 
   \Norm{B_{k \ell}}_2 \Norm{x_k}_2 \Norm{y_\ell}_2,\\
  &\leq \max_{k,\ell} \Norm{B_{k \ell}}_2
    \left(\sum_{(k,\ell): B_{k \ell}\not=0} \Norm{x_k}_2^2\right)^{\half}
    \left(\sum_{(k,\ell): B_{k \ell}\not=0} \Norm{y_\ell}_2^2\right)^{\half},\\
  &\leq  \gamma_C^\half\gamma_R^\half 
    \max_{k,\ell} \Norm{B_{k \ell}}_2 
    \left(\sum_{k} \Norm{x_k}_2^2\right)^{\half}
    \left(\sum_{\ell} \Norm{y_\ell}_2^2\right)^{\half}.
\end{align*}
This implies the assertion.
\end{proof}
The following lemma establishes estimates between the coefficient
vectors and the corresponding basis function expansions. This is a
standard result in finite element approximations and can also be
demonstrated using \Cref{lem:blocksparse}.
\begin{lemma}\label{lem:massmatrix}
There are constants $c_k>0$ such that
\begin{align*}
  f = \sum_{k m} f_{k m} \phi_{k m}
  & \quad \Rightarrow \quad 
     c_1 h_s  \Norm{\bf f}_2  \leq \Norm{f}_{L_2(\Gamma)} \leq c_2h_s  \Norm{\bf f}_2\\
  g = \sum_{i j} g_{ij} \chi_{i j}
  & \quad \Rightarrow  \quad
     c_3 h_t^\half \Norm{\bf g}_2  \leq \Norm{g}_{L_2(I)} \leq c_4 h_t^\half \Norm{\bf g}_2\\
  u = \sum_{k m i j} u_{k m i j} \phi_{km}\chi_{i j}
  & \quad \Rightarrow \quad                                    
     c_5 h_s h_t^\half \Norm{\bf u}_2  \leq \Norm{u}_{L_2(\Gamma\times I)} \leq c_6 h_s h_t^\half \Norm{\bf u}_2
  \end{align*}
\end{lemma}

\begin{lemma}\label{lem:Moments}
The moment matrices in \eqref{def:Mmat} satisfy the estimate
\begin{equation*}
  \Norm{M_\ell}_2 \lesssim \left( h_t p\log p  \, 2^{\ell} \right)^{\half} 
\end{equation*}
\end{lemma}
\begin{proof}
  For ${\bf q}\in \R^{N_{\ell,t}\times D_t}$ define the function
  $q := \sum_{i j} \chi_{ij} q_{i j}$, and for ${\bf f}\in \R^{p}$
  define the function $f := \sum_{\beta} L_\beta f_\beta$, then
\begin{equation*}
{\bf f}^T M_{\ell}{\bf q} = \int_{I^\ell_0} f(t) q(t) \,dt
\end{equation*}
From the Cauchy Schwarz inequality, it follows that
\begin{equation*}
  \abs{{\bf f}^T M_{\ell}{\bf q}} \leq
  \left( \int_{I^\ell_0} f^2(t) \,dt \right)^{\half}
  \left( \int_{I^\ell_0} q^2(t) \,dt \right)^{\half} .
\end{equation*}
The $q$-integral can be estimated using \Cref{lem:massmatrix}. The $f$-integral is transformed by a change of
variables to the standard interval $[-1,1]$. Then the resulting
integral can be estimated using the well-known fact that the Chebyshev
interpolation operator $\Pi_p$ satisfies the bound
$\Norm{\Pi_p}_\infty \lesssim \log p$, see, e.g. \cite{kress14},
Theorem 11.4. Thus
\begin{align*}
  \int_{I^\ell_0} f^2(t) \,dt &=
  \frac{h_t}{2} 2^\ell \int_{-1}^1 \widetilde f^2(t') \,dt' \leq
  h_t 2^\ell \Norm{\widetilde f}_{L_\infty[-1,1]}^2\\
 &\lesssim
 h_t 2^\ell \log p \Norm{\bf f}_\infty^2 \lesssim
 h_t  2^\ell p \log p \Norm{\bf f}_2^2 ,
\end{align*}
where $\widetilde f$ is the interpolate for the transplanted abscissas and
nodal values. Combining the estimates for $f$ and $q$ gives
the assertion.
\end{proof}

\subsection*{Temporal Far-field} We now turn to estimate the effect of the ACA
compression of the far-field matrices. To that end, assume that the ACA in
\eqref{aca:farfield} is performed until a certain level dependent accuracy
$\epsilon_\ell$ has been achieved, that is,
\begin{equation}
\Norm{E_d^\ell(\nu,\nu')}_2 \leq \epsilon_\ell. 
\end{equation}
To estimate the combined error of replacing all blocks in $A_d^\ell$ by the ACA approximation we use \Cref{lem:blocksparse}. Here 
$\gamma := \gamma_C = \gamma_R$ is the maximal number of neighbors of a cube
in level $\ell$,
\begin{equation}\label{est:nnbrs}
\gamma = \max_{\nu} \# \Nbrs(\nu).
\end{equation}
Thus it follows that
\begin{equation}
   \Norm{A_d^\ell - \hat A_d^\ell}_2 \leq \gamma \epsilon_\ell. 
\end{equation}
To estimate the resulting error of the ${\bf c}_d^\ell$ block, use
\eqref{c:m:a}. From \Cref{lem:Moments} and the fact that
$\Norm{A\otimes B}_2=\Norm{A}_2 \Norm{B}_2$ we get that
\begin{equation}\label{acaerr:far}
\Norm{{\bf c}_d^\ell - \hat {\bf c}_d^\ell}_2 \lesssim
  p\log p 2^{\ell}  h_t \epsilon_\ell.
\end{equation}

\subsection*{Temporal Near-field} 
Suppose now that the error of the ACA compression of a near-field block is
$\epsilon_{near}$, that is, $A_d(\nu,\nu') = \hat A_d(\nu,\nu') +
E_d(\nu,\nu')$, where $\Norm{E_d(\nu,\nu')}_2\leq \epsilon_{near}$. In
this case, we can conclude in a similar way as before that
\begin{equation}\label{acaerr:near}
  \Norm{{\bf c}^0_d - \hat {\bf c}_d^0}_2 \lesssim n_T \gamma \epsilon_{near},
\quad d \in \{0,1\}.
\end{equation}

\subsection*{Optimal Choice of the ACA parameter} From the 
triangle inequality and estimates \eqref{acaerr:far} and
\eqref{acaerr:near} one can estimate
\begin{equation}\label{preVhatV}
  \Norm{{\bf V} - \hat {\bf V}}_2 \lesssim \gamma \left( 
  n_T \epsilon_{near} +  p \log p 
  \;h_t \sum_{\ell=0}^{L-2} \epsilon_\ell 2^\ell 
  \right).
\end{equation}
Since $\epsilon_{near}$ and $\epsilon_\ell$ are at our disposition, we
choose a global accuracy $\epsilon$ and set
\begin{equation}\label{choose:eps}
\begin{aligned}
\epsilon_{near} &= \frac{\epsilon}{\gamma n_T } h_s h_t^{\half},\\
\epsilon_\ell &= \frac{\epsilon 2^{-\ell}}
                {\gamma p \log p \,(\ell+1)^2} h_s h_t^{-\half}.
\end{aligned}
\end{equation}
Then it follows from \eqref{preVhatV} that
\begin{equation*}
  \Norm{{\bf V} - \hat {\bf V}}_2 \lesssim  \epsilon h_s h_t^{\half}.
\end{equation*}
In the error analysis of integral equations estimates in
the energy norm are more important. In the case of the thermal 
single layer operator this is the $H^{-\half,-\quart}(\Gamma\times I)$
norm, see \cite{costabel90}. For functions in the discrete space $S$
this norm can bounded in terms of the $L_2$-norm using the inverse estimate
\begin{equation*}
\Norm{q}_{H^{-\half,-\quart}(\Gamma\times I)} \lesssim 
\left(h_s^{-\half} + h_t^{-\quart} \right)
\Norm{q}_{L_2(\Gamma\times I)} \lesssim 
h_s^{-1} h_t^{-\half} \left(h_s^{-\half} + h_t^{-\quart} \right) 
\Norm{\bf q}_2 
\end{equation*}
where the second step follows from \Cref{lem:massmatrix}. The last 
estimate immediately implies the main result of this section.
\begin{theorem}\label{theorem:acaerr}
  For $h_t\sim h_s^2$ and the ACA tolerances given in \eqref{choose:eps}, 
  the estimate
  \begin{equation*}
    \left \langle p, (\mathcal{V} - \mathcal{\hat V} )q \right\rangle
  = {\bf p}^T \left( {\bf V} - \hat {\bf V} \right) {\bf q}
    \lesssim \epsilon h_s^{-2} 
    \Norm{p}_{H^{-\half,-\quart}(\Gamma\times I)} 
    \Norm{q}_{H^{-\half,-\quart}(\Gamma\times I)}
  \end{equation*}
  holds for all $p,q \in S$. 
\end{theorem}
The effect of the error of the bilinear form on the solution can be 
estimated in a standard way using the Strang Lemma, see, e.g. \cite{braess07}, we thus obtain
\begin{equation*}
\Norm{q-q_h}_{H^{-\half,\quart}(\Gamma\times I)} \lesssim
\Norm{q-q_h}_{H^{-\half,\quart}(\Gamma\times I)} +
\epsilon h_s^{-2} \Norm{q}_{H^{-\half,\quart}(\Gamma\times I)} ,
\end{equation*}
where the first term is the discretization error and the second term is the
error of the fast method.

\subsection*{Complexity Estimates}
We use two complexity measures: first, the number of coefficients to store the
matrices $A_d^\ell$, $A_d$, and second, the number of floating point
operations to multiply these matrices with all vectors that are
necessary to evaluate \Cref{algo:bForwardElim} using the
sequence \eqref{fast:time:algor}.

For a matrix $A\in \R^{m\times n}$ that is stored in rank-$r$ format,
i.e., $A = U V^T$ with $U\in \R^{m\times r}$ and $V\in \R^{n\times r}$,
the asymptotic estimate of storage and matrix vector product
are $r(n+m)$. A well-known fact about compressing matrices in Galerkin 
methods with asymptotically smooth kernels is that the rank grows only
logarithmically with the accuracy $\epsilon$ see,
\cite{bebendorf08}. Since the heat kernel is smooth, the same result
applies, see \cite{watschinger22}. Denoting the maximal rank by
$r_{aca}$ then it follows from \eqref{choose:eps} that $r_{aca}$ grows
only logarithmically with the mesh sizes. 

Since only coefficients to neighboring blocks are stored 
we obtain in view of
\eqref{est:nnbrs},\eqref{est:patches} and \eqref{est:cubes}, 
for the storage of all $A_d^\ell$'s that
\begin{equation}\label{storage:far}
S_{far} \lesssim r_{aca} \sum_{\ell=0}^{L-2} \sum_{d=2}^3 \!\! \sum_{\nu \in \Cc_\ell \atop\nu'\in\Nbrs(\nu)}
\!\! \left( N_s(\nu) + N_s(\nu') \right)
\lesssim \gamma r_{aca} L N_s\,.
\end{equation}
To estimate the cost of a matrix vector product, one has to keep in
mind that the convolution structure of the matrix $\mathbf{V}$ implies
that the number of $\mathbf{c}^\ell_2$ and $\mathbf{c}^\ell_3$-blocks
is order $2^{L-\ell}$. Thus, a very similar calculation as before gives 
an estimate for the computational cost of the far-field
\begin{equation}\label{flops:far}
N_{far} \lesssim r_{aca} \sum_{\ell=0}^{L-2} \sum_{d=2}^3 2^{L-\ell} \!\! \sum_{\nu \in \Cc_\ell \atop\nu'\in\Nbrs(\nu)}
\!\!\left( N_s(\nu) + N_s(\nu') \right)
\lesssim \gamma r_{aca} L 2^L N_s = \gamma r_{aca} L N_t N_s.
\end{equation}

In the evaluation of the near-field, all cubes are in the finest spatial
level. Hence, the number of cubes is bounded by $N_s$, and the number of
patches per cube is bounded by a constant. Thus, even if no ACA
compression is used, the storage and the direct evaluation of the
near-field matrices scales like $\gamma N_s$ and $\gamma N_s N_t$,
respectively. However, the use of the ACA compression can significantly
reduce the constant factor in the asymptotic cost estimate.

\section{Numerical Results}\label{sec:numresults}
 We perform a numerical experiment on a benchmark problem to demonstrate the effectiveness of the hybrid ACA compression. We have implemented all the algorithms in \texttt{C} and performed the experiments on an Intel Xeon Gold $2.5$GHz processor.  We first provide some more details about the aspects of the overall solution scheme that have not been discussed so far.

\subsection*{Linear solver} Every time step in the forward elimination 
\eqref{forward:elimination} involves solving the system $A_0 q_i =  b_i$.
Since the matrix $A_0$ has an exponentially decaying kernel it is stored in 
sparse format. Further, $A_0$ is symmetric positive
definite and has bounded condition number when $h_t \sim h_s^2$, see 
\cite{messner-etal15}. Thus, we simply use unpreconditioned conjugate gradients (CG) to solve the linear system. 

\subsection*{Quadrature for matrix coefficents}
The coefficients of the matrices $A_d$, and $A_d^ \ell$ involve surface integrals over pairs of triangular patches. These integrals are mapped on four dimensional cubes and calculated using tensor product Gauss Legendre rules. In the case $d=0$ and $d=1$, the kernel or its derivatives have singularities for adjacent or identical patches. Here well-known singularity removing transformations for elliptic Galerkin BEM can be used to ensure rapid convergence of the quadrature rules with respect to the quadrature order, see \cite{sauter-schwab11}. 

\subsection*{Choice of parameters} 
The main parameters that control the convergence of the hybrid ACA method are the separation ratio $\eta$ in \eqref{def-sepRatio}, the ACA tolerance $\epsilon$ in \eqref{choose:eps}, the order of the temporal Chebychev expansion $p$ in \eqref{interpol:G} and the quadrature order $p_Q$. The goeal is to adjust these parameters to $h_s$ and $h_t$ such that the convergence rate of the direct method is maintained by the accelerated scheme.  

The separation ratio controls how many neighbors are retained in the matrices $A_d^\ell$ and $A_d$ and thus the factor $\gamma$ in \eqref{est:nnbrs}. The scaling of the time and spatial clustering implies that $\eta$ can be the same in all spatial levels. However, when refining the mesh $\eta$ must be decreased at an algebraic rate to decrease the spatial far-field truncation at an exponential rate. This implies that $\gamma$ increases algebraically, contributing a logarithmic factor in the overall complexity estimate.

\Cref{theorem:acaerr} determines how the ACA tolerance has to be chosen. Because of the factor $h_s^{-2}$ in the estimate, the parameter must be $\epsilon \sim h_s^{2+r}$ when the error of the direct scheme is $h_s^r$.
Thus, the ranks of the ACA approximations will increase as $h_s\to 0$ contributing a logarithmic term in the complexity. Likewise, the Chebyshev order $p$ must be increased, contributing another logarithm in the complexity.

Finally, adjusting the quadrature order to the desired convergence rate contributes to another logarithmic term to the time to set up all matrices. 

\subsection{Example Problem}
We consider $\Gamma$ as the surface of a unit sphere for the time interval $[0,1]$. We have created a right hand side for which the analytic solution is a second order polynominal in space and time.
The surface of the sphere is initially triangulated into $48$ triangular patches. Then refinements are created by subdividing each triangle into four congruent triangles, see \Cref{fig:sphere}. Here, $N_s$ and $N_t$ represent the number of triangular patches and time intervals, respectively. 

\begin{figure}[H]
    \centering
    \subfloat[]{\includegraphics[height=5cm, width=5.5cm]{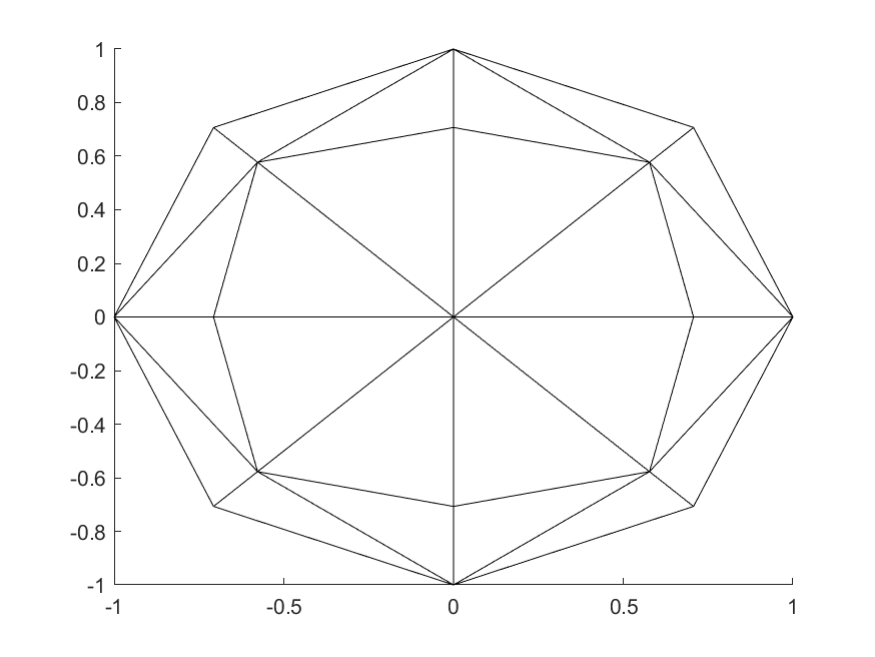}\label{sp0}} \quad
    \subfloat[]{\includegraphics[height=5cm, width=5.5cm]{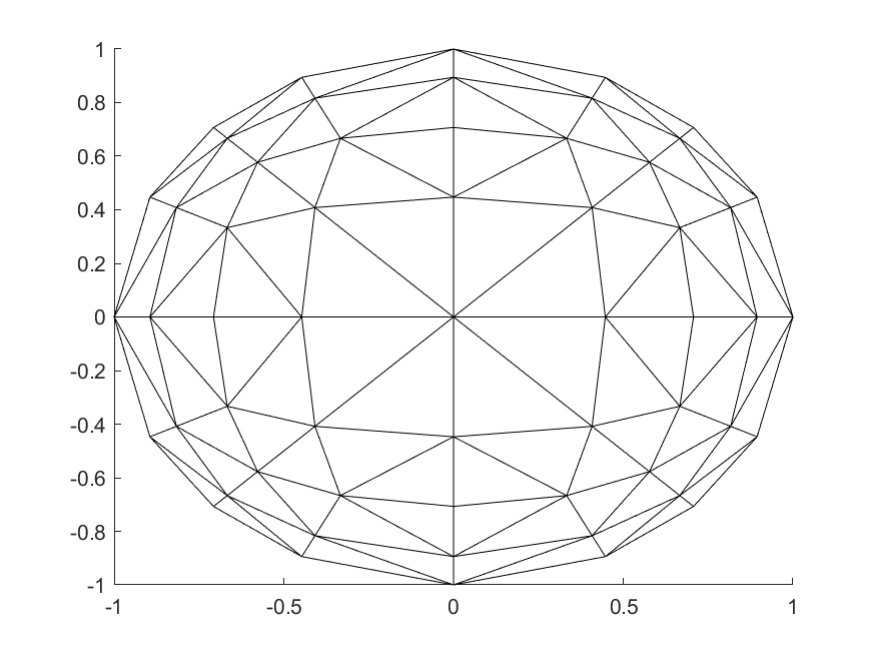}\label{sp1}}%
    \caption{Triangulation of the unit sphere and its refinements.}
    \label{fig:sphere}
\end{figure}

We perform two experiments with two different nodal shape functions. In the first experiment, we assume the nodal shape function in space and time as piecewise constant elements $(a=0)$.  In the second experiment, we assume piecewise constants elements in time and piecewise linear elements in space $(a=1)$. In the latter experiment, we use the extension and restriction method of \Cref{sec:contin} to obtain continuous elements.

In \Cref{tab:parameters_piecewise_constant} and \Cref{tab:parameters_piecewise_linear}, we tabulate all the parameters used in the experiments for $a=0$ and $a=1$, respectively. 
The approximation rate for the constant elements is $O(h_s)$, thus according to \Cref{theorem:acaerr}, the ACA tolerance parameter $\epsilon$ should scale like $O(h_s^3)$. In our experiments we reduce this parameter by a factor of ten with each refinement. This is slightly more than the theoretical factor of eight. For $a=1$, the approximation power is $O(h_s^2$), thus $\epsilon$ should be reduced by a factor of 16, which is what we did in our calculations. Note that we use only a second order quadrature rule for $a=0$ and both second and third order quadrature rules for $a=1$. The order of the temporal Chebychev expansion $p$ appears to have a lesser effect on the overall error. In all experiments, $p=4$ was sufficient, except in the highest accuracy case, where the value had to be increased (\Cref{tab:parameters_piecewise_linear}).

\begin{table}[H]
     \centering
     \resizebox{9cm}{!}{
    \begin{tabular}{|l|l|l|l|l|l|l|l|}
    \hline
     $\epsilon$ & $L_s$  & $L$ & $\eta$ & $p_Q$ & $p$ & $N_s$ &  $N_t$  \\ \hline
     2e-02  & 2 & 3 & 0.40  & 2 & 4& 48  & 40     \\ 
     2e-03  & 2 & 5 & 0.39  & 2 & 4& 192 & 160    \\ 
     2e-04  & 3 & 7 & 0.36  & 2 & 4& 768 & 640    \\ 
     2e-05  & 4 & 9 & 0.33  & 2 & 4& 3072 & 2560  \\ 
     2e-06  & 5 & 11& 0.30  & 2 & 4& 12288 & 10240 \\ \hline
    \end{tabular}}
    \caption{Parameters used in experiment for piecewise constant element $(a=0)$. The relation between the spatial level and the temporal level is given by \cref{level_relation}.}
    \label{tab:parameters_piecewise_constant}
\end{table}

\begin{table}[H]
     \centering
     \resizebox{9cm}{!}{
    \begin{tabular}{|l|l|l|l|l|l|l|l|}
    \hline
     $\epsilon$ & $L_s$  & $L$ & $\eta$ & $p_Q$ & $p$ & $N_v$ & $N_t$  \\ \hline
     4e-03     & 2 & 3 & 0.40  & 2, 3 &  4& 26   & 40   \\
     1.6e-04   & 2 & 5 & 0.35  & 2, 3 &  4& 98   & 160  \\ 
     1e-05     & 3 & 7 & 0.30  & 2, 3 &  4& 386  & 640 \\
     6.25e-07  & 4 & 9 & 0.25  & 2, 3 &  5& 1538 & 2560 \\ \hline
    \end{tabular}}
    \caption{Parameters used in experiment for piecewise linear element $(a=1)$. The relation between the spatial level and the temporal level is given by \cref{level_relation}.}
    \label{tab:parameters_piecewise_linear}
\end{table}

 \Cref{fig:setup} displays the setup time for all matrices $\hat A_d^\ell$ and $\hat A_d$ in the ACA representation of {\bf V}. This includes the quadrature time for the matrix coefficients as well as the time to calculate the matrices $U$ and $V$ in the low-rank representation. 
\begin{figure}[H]
    \centering
    \subfloat[$a=0$]{\includegraphics[height=4cm, width=6.25cm]{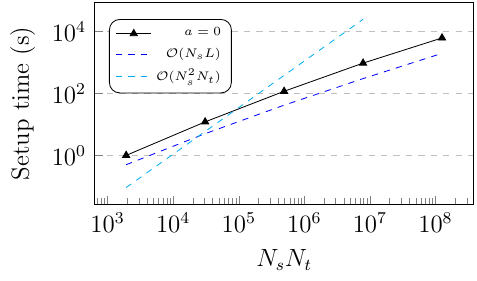}\label{setup0}} \quad
    \subfloat[$a=1$]{\includegraphics[height=4cm, width=6.25cm]{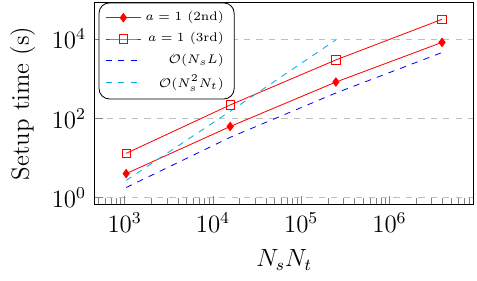}\label{setup1}}%
    \caption{Setup time (sec) versus the $N_s N_t$. The quadrature order is mentioned in the parenthesis for $a=1$.}
    \label{fig:setup}
\end{figure}

We plot the time to solve $\bf V \bf q = \bf p$ using \Cref{algo:bForwardElim} in \Cref{fig:solution}. 
\begin{figure}[H]
    \centering
    \subfloat[$a=0$]{\includegraphics[height=4cm, width=6.25cm]{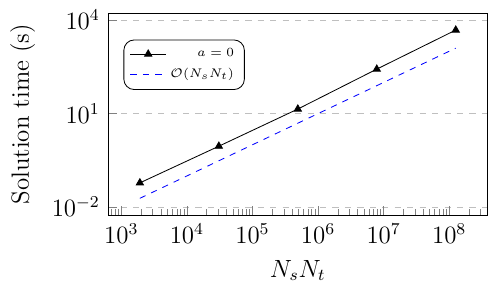}\label{solution0}} 
    \quad
    \subfloat[$a=1$]{\includegraphics[height=4cm, width=6.25cm]{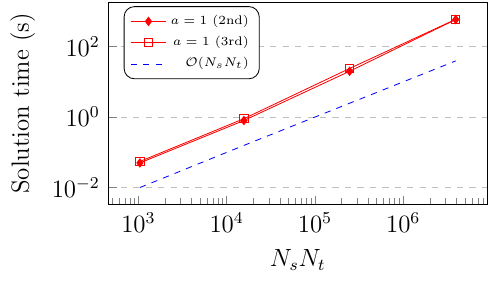}\label{solution1}}%
    \caption{Solution time (sec) versus the $N_s N_t$.}
    \label{fig:solution}
\end{figure}

Both setup and solution times show the theoretical complexity estimates of \Cref{sec:analysis} well. For comparison, we also include a curve for $O(N_s^2 N_t)$, which is the estimated complexity of setting up the uncompressed matrix 
\eqref{toeplitz:uncompressed} while exploiting the block Toeplitz structure.

Note that even though the setup time has a lower asymptotic cost, it dominates over the solution time for the meshes we were able to compute. Only for much finer meshes than shown here, the solution time will be longer than the setup time. This has to do with the high cost of evaluating the integrals for the matrix coefficients. We have shown data for second and third order quadrature rules to illustrate the effect of increasing the order.

% \begin{figure}[H]
%     \centering
%     \subfloat[$a=0$]{\includegraphics[height=4cm, width=6.25cm]{Plots/total_time0.pdf}\label{total_time0}}
%     \quad
%     \subfloat[$a=1$]{\includegraphics[height=4cm, width=6.25cm]{Plots/total_time1.pdf}\label{total_time1}}
%     \caption{Total time (setup time $+$ solution time) versus the $N_s N_t$.}
%     \label{fig:total_time}
% \end{figure}

\Cref{fig:error} displays the $L_2$-norm of the absolute error in the solution. 
For $a=0$, the same quadrature order $p_Q=2$ is sufficient to reproduce the optimal $O(h_s)$ behavior. In the case $a=1$, the finest mesh requires increasing the quadrature to $p_Q=3$ to preserve the optimal $O(h_s^2)$ error. It is worth noting that $a=1$ gives better approximations to the solution compared to $a=0$.

\begin{figure}[H]
    \centering
    \subfloat[$a=0$]{\includegraphics[height=4cm, width=6.25cm]{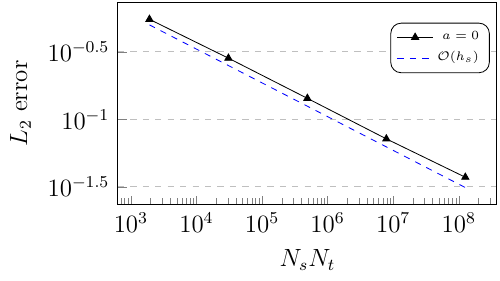}\label{error0}}%
    \quad
    \subfloat[$a=1$]{\includegraphics[height=4cm, width=6.25cm]{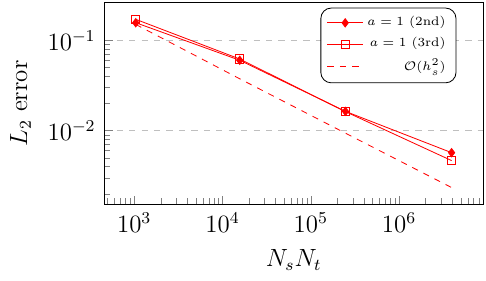}\label{error1}}%
    \caption{$L_2$-norm of the absolute error in the solution versus the $N_s N_t$.}
    \label{fig:error}
\end{figure}

Now, we illustrate the effectiveness of the ACA. We denote the ACA approximation of the temporal far-field block matrix (at the temporal level $l$) $A_d^\ell$ as $\hat A_d^\ell$. We compare the cumulative number of entries in the matrices $A_d^\ell$ with those in the matrices $\hat A_d^\ell$ across all the temporal levels. \Cref{fig:far} displays the number of entries of $A_d^\ell$ (no ACA) and those in $\hat A_d^\ell$ (with ACA). We denote the ACA approximation of the temporal near-field block matrix $A_d$ as $\hat A_d$ with $d \in \{0,1\}$. We compare the cumulative number of entries in the matrices $A_d$ with those in the matrices $\hat A_d$. \Cref{fig:near} presents the number of entries of $A_d$ (no ACA) and those in $\hat A_d$ (with ACA). Here, we only report the $2$nd order quadrature rule for $a=1$ because the plots for both the $2$nd and $3$rd order quadrature rules overlap. From \Cref{fig:far,fig:near}, it is clear that we get better compression as the value of $N_sN_t$ grows. Therefore, ACA-based compression becomes highly effective for large $N_s N_t$.

\begin{figure}[H]
    \centering
    \subfloat[$a=0$]{\includegraphics[height=4cm, width=6.25cm]{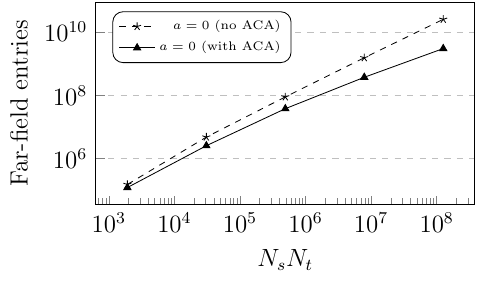}\label{far0}} \quad
    \subfloat[$a=1$]{\includegraphics[height=4cm, width=6.25cm]{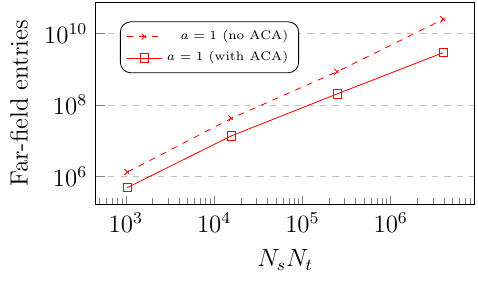}\label{far1}}%
    \caption{Number of far-field entries versus the $N_s N_t$.}
    \label{fig:far}
\end{figure}

\begin{figure}[H]
    \centering
    \subfloat[$a=0$]{\includegraphics[height=4cm, width=6.25cm]{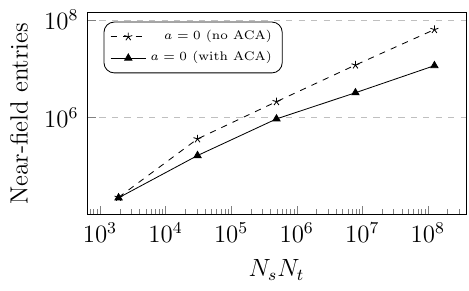}\label{near0}} \quad
    \subfloat[$a=1$]{\includegraphics[height=4cm, width=6.25cm]{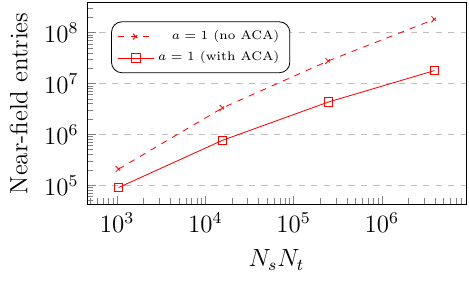}\label{near1}}%
    \caption{Number of near-field entries versus the $N_s N_t$.}
    \label{fig:near}
\end{figure}

\section{Conclusions}
We have presented a new ACA-based compression approach for discretizations of thermal potential operators. We have shown analytically and numerically how to control the various parameters. The resulting scheme has log-linear complexity and preserves the approximation power of the discretization space. The ACA compression method has the advantage that its formulation only involves linear algebra and thus does not depend on the specific form of the heat kernel. Thus, it should be easily adaptable to more complicated parabolic problems, such as transient Stokes flow. A preliminary version of this work was published in \cite{wang20}. This dissertation contains further numerical illustrations of more realistic heat transfer problems in complicated geometries.

%%%%%%%%%%%%%%%% Acknowledgments %%%%%%%%%%%%%%%
\section*{Acknowledgments}
The authors acknowledge the use of the computing resources at HPCE, IIT Madras.
%%%%%%%%%%%%%%%%%%%%%%%%%%%%%%%%%%%%%%%%%%%%%%%%

%\bibliography{mrabbrev,references}
%\bibliographystyle{plain}

\end{document}